\documentclass[leqno]{article}
\usepackage{geometry}
\usepackage{graphicx}
\usepackage[utf8]{inputenc}
\usepackage[T1]{fontenc}
\usepackage[english]{babel}
\usepackage{mathtools}
\usepackage{amsfonts,amssymb,mathrsfs,amscd,amsmath,amsthm}
\usepackage[usenames,dvipsnames,svgnames,table]{xcolor}
\usepackage{verbatim}
\usepackage{url}
\usepackage{hyperref}
\usepackage{stmaryrd}

\theoremstyle{plain}

\newtheorem{thm}{Theorem}[section]
\newtheorem{prop}[thm]{Proposition}
\newtheorem{cor}[thm]{Corollary}

\theoremstyle{definition}

\newtheorem{rk}[thm]{Remark}

\def\:{\colon}
\def\rom#1{\emph{#1}}
\def\({\rom(}
\def\){\rom)}

\def\dl{\delta}
\def\e{\varepsilon}

\def\dis{\operatorname{dis}}
\def\codis{\operatorname{codis}}
\def\diam{\operatorname{diam}}
\def\gr{\operatorname{Gr}}

\def\cl#1{\overline {#1}}

\begin{document}
\title{Surjective Mappings in the Hyers--Ulam Theorem and the Gromov--Hausdorff Distance\footnote{The work of A.A. Tuzhilin was supported by grant No. 25-21-00152 of the Russian Science Foundation, by National Key R\&D Program of China (Grant No. 2020YFE0204200), as well as by the Sino-Russian Mathematical Center at Peking University. Partial of the work by A.A. Tuzhilin were done in Sino-Russian Math. center, and he thanks the Math. Center for the invitation and the hospitality.}}
\author{S.\,A.~Bogaty\u{\i}, E.\,A.~Reznichenko, A.\,A.~Tuzhilin}
\date{}
\maketitle
\tableofcontents

\begin{abstract}
A topological space is said to be cardinality homogeneous if every nonempty open subset has the same cardinality as the space itself. Let $X$ and $Y$ be cardinality homogeneous metric spaces of the same cardinality. If there exists a $\dl$-surjective $d$-isometry between such equicardinal cardinality homogeneous metric spaces $X$ and $Y$, then there exists a bijective $(d+2\dl)$-isometry between $X$ and $Y$. This result allows us to reduce the Dilworth--Tabor theorem to the Gevirtz--Omladi\v{c}--\v{S}emrl theorem on approximation by isometries and, in particular, to questions concerning the isometry of Banach spaces.

\end{abstract}


\section{Introduction}
The celebrated Mazur--Ulam theorem of 1932 states that \emph{every surjective isometry $f\colon V\to W$ between real normed spaces is affine}~\cite{MU32}. One direction for strengthening this theorem arose from the 1945 work of Hyers and Ulam~\cite{HU45}. D.\,H.~Hyers and S.\,M.~Ulam posed the following question:

\emph{Is there a constant $K$, depending only on the Banach spaces $V$ and $W$, such that for every $d>0$ and every surjective $d$-isometry $f\colon V\to W$ there exists a surjective isometry $U\colon V\to W$ satisfying}
$$
\|f - U\| = \sup\bigl\{\|f(x)-U(x)\| : x\in V\bigr\} \le K\,d?
$$
They noted that \emph{the surjectivity assumption for $f$ is essential} and provided an affirmative answer in the case $V=W=\ell_2$, the Hilbert space (with $K=10$).

The efforts of numerous mathematicians over many years culminated in a complete affirmative resolution by J.~Gevirtz in 1983~\cite{G83}. In 1995, M.~Omladi\v{c} and P.~\v{S}emrl~\cite{OS95} established the sharp constant $K=2$.

S.\,J.~Dilworth~\cite{D99} and J.~Tabor~\cite{T00} relaxed the surjectivity condition:

\emph{If $d,\dl \ge 0$ and $f\colon V\to W$ is a $d$-isometry between Banach spaces with $f(0)=0$ that maps $V$ $\dl$-surjectively onto a closed subspace $L\subset W$ \(i.e., $d_H(L,f(V))\le\dl$\), then there exists a bijective linear isometry $U\colon V\to L$ such that}
\begin{equation}\label{eq:d1}
\|f-U\| \le M=12d+5\dl \quad \text{(Dilworth)}; \qquad M=2d+35\dl \quad \text{($L=W$, Tabor)}.
\end{equation}
In 2003, P.~\v{S}emrl and J.~V\"{a}is\"{a}l\"{a}~\cite{SV03} showed that \emph{the right-hand side of $(\ref{eq:d1})$ can be replaced by $M=2d+2\dl$}. They also proved that \emph{when $L=W$, the universal bound $M=2d$ holds}.

Inspired by the proof technique of the Cantor--Bernstein theorem, the present paper approximates two mappings $f\colon X\to Y$ and $g\colon Y\to X$ between sets by a bijection (Theorem~\ref{thm:1-1exist}).

Two applications of this approximation are presented. First, we show that for cardinality homogeneous metric spaces of the same cardinality, the Gromov--Hausdorff distance can be characterised using bijective mappings (Theorem~\ref{cor:dis}).

Theorem~\ref{thm:1-1exist} and its corollaries also permit a reduction of the approximation problem for $d$-isometric $\dl$-surjective mappings of Banach spaces (onto a closed linear subspace) to the approximation of $(d+2\dl)$-isometric (or $(d+6\dl)$-isometric) surjective mappings (see the general Theorems~\ref{thm:sur} and~\ref{thm:sur2} and Propositions~\ref{prop:bsh} and~\ref{prop:bshL} for Banach spaces). This yields a simpler reduction, in our view, of the Dilworth--Tabor theorem to the theorem of J.~Gevirtz (see Corollaries~\ref{cor:Tabor} and~\ref{cor:bshL}). A further application of the \v{S}emrl--V\"{a}is\"{a}l\"{a} theorem shows that the coefficient $d+2\dl$ arising in this reduction is non-essential.

\section{Preliminaries and Notation}
Let $(X,d)$ be a metric space (distances may be infinite). For $\e>0$, $x,y\in X$, and $M\subset X$, we denote by
\begin{align*}
|xy| &= d(x,y) && \textit{the distance between $x$ and $y$}; \\
U_\e(x) &= \{y\in X : |xy| < \e\} && \textit{the open $\e$-ball centered at $x$}; \\
U_\e(M) &= \bigcup_{x\in M} U_\e(x) && \textit{the $\e$-neighborhood of $M$}; \\
\diam M &= \sup\{|xy| : x,y\in M\} && \textit{the diameter of $M$}.
\end{align*}
For subsets $A,B\subset X$, the \emph{Hausdorff distance\/} is
$$
d_H(A,B) = \inf\{\e>0 : A\subset U_\e(B)\ \text{and}\ B\subset U_\e(A)\}.
$$
The \emph{Gromov--Hausdorff distance $d_{GH}(X,Y)$} between metric spaces $X$ and $Y$ is the infimum of $d_H(X,Y)$ over all isometric embeddings of $X$ and $Y$ into a common metric space $Z$; see~\cite{E75,G81,G99} for details.

A \emph{correspondence\/} between sets $X$ and $Y$ is a relation $R\subset X\times Y$ such that every $x\in X$ is related to at least one $y\in Y$, and vice versa.

For nonempty $R\subset X\times Y$, its \emph{distortion\/} is
$$
\dis R = \sup\bigl\{ \bigl||xx'|-|yy'|\bigr| : (x,y),(x',y')\in R \bigr\}.
$$
An equivalent expression for the Gromov--Hausdorff distance is~\cite[Theorem 7.3.25]{BBI}
$$
d_{GH}(X,Y) = \frac{1}{2} \inf\{\dis R : R\subset X\times Y\ \text{is a correspondence}\}.
$$
The \emph{co-distortion\/} of nonempty $R,R'\subset X\times Y$ is
$$
\codis(R,R') = \sup\bigl\{ \bigl||xx'|-|yy'|\bigr| : (x,y)\in R,\ (x',y')\in R' \bigr\}.
$$
For mappings $f\colon X\to Y$ and $g\colon Y\to X$, the \emph{distortion\/} of $f$ is
$$
\dis f = \sup\bigl\{ \bigl||f(x)f(x')|-|xx'|\bigr| : x,x'\in X \bigr\},
$$
and the \emph{co-distortion\/} is
$$
\codis(f,g) = \sup\bigl\{ \bigl||xg(y)| - |f(x)y|\bigr| : x\in X,\ y\in Y \bigr\}.
$$
The \emph{graph\/}\footnote{In set theory, the graph of a mapping is identified with the mapping itself. We use $\gr f$ for clarity.} of $f$ is
$$
\gr f = \{(x,f(x)) : x\in X\}.
$$
For a relation $Q\subset Y\times X$, we denote $Q^{-1}=\{(x,y)\in X\times Y: (y,x)\in Q\}$ the {\em converse relation}. We denote
\[
\gr^{-1} g = (\gr g)^{-1} = \{ (g(y),y) : y\in Y \}.
\]
The following facts are easily verified:
\begin{itemize}
\item $R\subset X\times Y$ is a correspondence if and only if there exist $f\in Y^X$ and $g\in X^Y$ with $\gr f \cup \gr^{-1} g \subset R$;
\item $\dis f = \dis(\gr f)$;
\item $\codis(f,g) = \codis(\gr f, \gr^{-1} g)$;
\item $\dis(R\cup R') = \max\{\dis R, \codis(R,R'), \dis R'\}$.
\end{itemize}
Thus,
$$
d_{GH}(X,Y) = \inf\{ d_{GH}(f,g) : f\in Y^X,\ g\in X^Y\},
$$
where
$$
d_{GH}(f,g) = \frac{1}{2} \max\{\dis f, \codis(f,g), \dis g\}.
$$
This form is noted in~\cite[Proposition 1(5), Remark 2]{MS2005} for bounded spaces.

The \emph{density\/} of a topological space $X$ is the cardinal
$$
d(X) = \min\{|D| : D\subset X = \cl D\}.
$$

\section{Approximation of Mappings by a Bijection}\label{sec:amb}
\begin{thm}
\label{thm:1-1exist}
Let $X$ and $Y$ be sets equipped with covers $\mathcal{L}_X$ and $\mathcal{L}_Y$ such that $|X|=|Y|=|V|$ for every $V\in\mathcal{L}_X$ and $|U|$ for every $U\in\mathcal{L}_Y$. Then for any mappings $f\colon X\to Y$ and $g\colon Y\to X$, there exists a bijection $\tilde{f}\colon X\to Y$ such that for every $x\in X$ at least one of the following holds\/\rom:
\begin{enumerate}
\item $\tilde{f}(x)$ and $f(x)$ lie in the same member of $\mathcal{L}_Y$\rom;
\item $g(\tilde{f}(x))$ and $x$ lie in the same member of $\mathcal{L}_X$.
\end{enumerate}
\end{thm}
\begin{proof}
Choose functions $U\colon X\to\mathcal{L}_X$ and $V\colon Y\to\mathcal{L}_Y$ with $x\in U(x)$ and $y\in V(y)$. Enumerate the points using the cardinal $\tau = |X|=|Y|$: $X = \{x_\alpha : \alpha < \tau\}$ and $Y = \{y_\alpha : \alpha < \tau\}$. For $\alpha \le \tau$, set $X_\alpha = \{x_\beta : \beta < \alpha\}$ and $Y_\alpha = \{y_\beta : \beta < \alpha\}$.

By transfinite induction on $\alpha \le \tau$, construct injective mappings $f_\alpha\colon X_\alpha \to Y$ and $g_\alpha\colon Y_\alpha \to X$ satisfying, for every $\beta < \alpha$:
\begin{itemize}
\item $f_\alpha|_{X_\beta} = f_\beta$ and $g_\alpha|_{Y_\beta} = g_\beta$;
\item if $f_\alpha(x_\beta) \in Y_\alpha$, then $x_\beta = g_\alpha(f_\alpha(x_\beta))$;
\item if $g_\alpha(y_\beta) \in X_\alpha$, then $y_\beta = f_\alpha(g_\alpha(y_\beta))$;
\item either $f_\alpha(x_\beta) \in V(f(x_\beta))$ or $g_\alpha(y_\beta) \in U(g(y_\beta))$.
\end{itemize}

If $\alpha$ is limit, take unions of the previous mappings. For successor $\alpha = \beta + 1$, extend $f_\alpha$ and $g_\alpha$ as follows.

To define $f_\alpha(x_\beta)$:
\begin{itemize}
\item[Case 1:] If $x_\beta \notin g_\beta(Y_\beta)$, choose $y \in V(f(x_\beta)) \setminus (Y_\beta \cup f_\beta(X_\beta))$ and set $f_\alpha(x_\beta) = y$.
\item[Case 2:] If $x_\beta \in g_\beta(Y_\beta)$, set $f_\alpha(x_\beta) = y$ where $y \in Y_\beta$ and $x_\beta = g_\beta(y)$.
\end{itemize}

To define $g_\alpha(y_\beta)$:
\begin{itemize}
\item[Case 1:] If $y_\beta \notin f_\alpha(X_\alpha)$, choose $x \in U(g(y_\beta)) \setminus (X_\alpha \cup g_\beta(Y_\beta))$ and set $g_\alpha(y_\beta) = x$.
\item[Case 2:] If $y_\beta \in f_\alpha(X_\alpha)$, set $g_\alpha(y_\beta) = x$ where $x \in X_\alpha$ and $y_\beta = f_\alpha(x)$.
\end{itemize}

The construction yields mutually inverse bijections $\tilde{f} = f_\tau$ and $g_\tau$ satisfying the required conditions for all $x\in X$.
\end{proof}
We write $x \in X_I$ if condition (1) holds for $x$, and $x \in X_{II}$ if (2) holds.

\section{Pseudometrics on Classes of Metric Spaces}\label{sec:pmcms}
Let metrics be defined on $X$ and $Y$, with $f\colon X\to Y$ and $g\colon Y\to X$. Define
\begin{align*}
d_f^- &= \sup\{|xx'| - |f(x)f(x')| : x,x'\in X\}, & d_f^+ &= \sup\{|f(x)f(x')| - |xx'| : x,x'\in X\}, \\
d_{(f,g)}^- &= \sup\{|x g(y)| - |f(x) y| : x\in X, y\in Y\}, & d_{(f,g)}^+ &= \sup\{|f(x) y| - |x g(y)| : x\in X, y\in Y\}.
\end{align*}
Then
\begin{gather}
-d_f^- \le |f(x)f(x')| - |xx'| \le d_f^+ \quad \text{for all $x,x'\in X$},\label{eq:d2}\\
-d_g^- \le |g(y)g(y')| - |yy'| \le d_g^+ \quad \text{for all $y,y'\in Y$},\label{eq:d3}\\
-d_{(f,g)}^- \le |f(x) y| - |x g(y)| \le d_{(f,g)}^+ \quad \text{for all $x\in X$, $y\in Y$}.\label{eq:d4}
\end{gather}
Clearly,
$$
d_{(g,f)}^+ = d_{(f,g)}^- \quad \text{and} \quad d_{(g,f)}^- = d_{(f,g)}^+.
$$
We have
\begin{gather*}
\dis f = \max\{d_f^+, d_f^-\}, \quad \codis(f,g) = \max\{d_{(f,g)}^+, d_{(f,g)}^-\}, \\
d_{GH}(f,g) = \frac{1}{2} \max\{d_f^-, d_f^+, d_g^-, d_g^+, d_{(f,g)}^+, d_{(f,g)}^-\}.
\end{gather*}
A mapping $f$ is a $d$-isometry if $\dis f \le d$. One-sided bounds are also useful, e.g., $d_f^+$-non-expansive mappings~\cite{V17}.

Define
$$
md_{GH}(f,g) = \frac{1}{2} \max\{\dis f, \dis g\} = \frac{1}{2} \max\{d_f^-, d_f^+, d_g^-, d_g^+\}.
$$
Evidently, $md_{GH}(f,g) \le d_{GH}(f,g)$.

For a class $\mathcal{F} \subset Y^X \times X^Y$, set
\begin{align*}
d_{GH}(X,Y;\mathcal{F}) &= \inf\{d_{GH}(f,g) : (f,g)\in\mathcal{F}\}, \\
md_{GH}(X,Y;\mathcal{F}) &= \inf\{md_{GH}(f,g) : (f,g)\in\mathcal{F}\}.
\end{align*}
\begin{prop}\label{prop:dgh-leq}
If $\mathcal{F}' \subset \mathcal{F}$, then
\begin{align*}
md_{GH}(X,Y;\mathcal{F}) &\le d_{GH}(X,Y;\mathcal{F}) \le d_{GH}(X,Y;\mathcal{F}'), \\
md_{GH}(X,Y;\mathcal{F}) &\le md_{GH}(X,Y;\mathcal{F}').
\end{align*}
\end{prop}
The standard Gromov--Hausdorff distance is
$$
d_{GH}(X,Y) = d_{GH}(X,Y; Y^X \times X^Y).
$$
The modified Gromov--Hausdorff distance~\cite{M12}, which is algorithmically simpler, is
\begin{equation}\label{eq:d6}
md_{GH}(X,Y) = md_{GH}(X,Y; Y^X \times X^Y).
\end{equation}
Define
\begin{align*}
\mathcal{M}_c(X,Y) &= \{f\in Y^X : f\ \text{ is continuous}\}, \\
\mathcal{F}_c(X,Y) &= \mathcal{M}_c(X,Y) \times \mathcal{M}_c(Y,X).
\end{align*}
The \emph{continuous Gromov--Hausdorff distance}~\cite{LMS23, BT26} is
\begin{equation}\label{eq:d5}
d_{GH}^c(X,Y) = d_{GH}(X,Y; \mathcal{F}_c(X,Y)),
\end{equation}
and its \emph{modified version\/} is
$$
md_{GH}^c(X,Y) = md_{GH}(X,Y; \mathcal{F}_c(X,Y)).
$$
From Proposition~\ref{prop:dgh-leq} we obtain
\begin{align*}
md_{GH} \le d_{GH} \le d_{GH}^c, \qquad md_{GH} \le md_{GH}^c.
\end{align*}

\subsection{Distance Between Spaces of the Same Cardinality}\label{ssec:peq}
Let $X$ and $Y$ be metric spaces of the same cardinality. Define
\begin{align*}
\mathcal{M}_{\mathrm{in}}(X,Y) &= \{f\colon X\to Y : f\ \text{is injective}\}, &
\mathcal{M}_{\mathrm{sur}}(X,Y) &= \{f\colon X\to Y : f\ \text{is surjective}\}, \\
\mathcal{M}_{\mathrm{bi}}(X,Y) &= \{f\colon X\to Y : f\ \text{is bijective}\}, \\
\mathcal{F}_{\mathrm{in}}(X,Y) &= \mathcal{M}_{\mathrm{in}}(X,Y) \times \mathcal{M}_{\mathrm{in}}(Y,X), &
\mathcal{F}_{\mathrm{sur}}(X,Y) &= \mathcal{M}_{\mathrm{sur}}(X,Y) \times \mathcal{M}_{\mathrm{sur}}(Y,X), \\
\mathcal{F}_{\mathrm{bi}}(X,Y) &= \mathcal{M}_{\mathrm{bi}}(X,Y) \times \mathcal{M}_{\mathrm{bi}}(Y,X), &
\mathcal{F}_i(X,Y) &= \{(f,f^{-1}) : f\in\mathcal{M}_{\mathrm{bi}}(X,Y)\}.
\end{align*}
The corresponding (modified) distances are
\begin{align}\label{eq:d7}
(m)d_{GH}^{\mathrm{sur}}(X,Y) &= (m)d_{GH}(X,Y; \mathcal{F}_{\mathrm{sur}}(X,Y)), \notag \\
(m)d_{GH}^{\mathrm{in}}(X,Y) &= (m)d_{GH}(X,Y; \mathcal{F}_{\mathrm{in}}(X,Y)), \notag \\
(m)d_{GH}^{\mathrm{bi}}(X,Y) &= (m)d_{GH}(X,Y; \mathcal{F}_{\mathrm{bi}}(X,Y)).
\end{align}
Also set
$$
(m)d_{GH}^i(X,Y) = (m)d_{GH}(X,Y; \mathcal{F}_i(X,Y)).
$$
From Proposition~\ref{prop:dgh-leq},
$$
md_{GH}^{\mathbf{xx}} \le d_{GH}^{\mathbf{xx}} \quad \text{for $\mathbf{xx} \in \{i,\mathrm{in},\mathrm{sur},\mathrm{bi}\}$}
$$
and
\begin{align*}
(m)d_{GH} &\le (m)d_{GH}^{\mathrm{in}} \le (m)d_{GH}^{\mathrm{bi}} \le (m)d_{GH}^i, \\
(m)d_{GH} &\le (m)d_{GH}^{\mathrm{sur}} \le (m)d_{GH}^{\mathrm{bi}}.
\end{align*}
\begin{prop}\label{prop:1-1equal}
For a bijection $f\colon X\to Y$ between metric spaces,
$$
d_{f^{-1}}^+ = d_f^-, \quad d_{f^{-1}}^- = d_f^+; \quad d_{(f,f^{-1})}^+ = d_f^+, \quad d_{(f,f^{-1})}^- = d_f^-.
$$
Consequently,
$$
md_{GH}(f,f^{-1}) = d_{GH}(f,f^{-1}) = \frac{1}{2} \dis f = \frac{1}{2} \dis f^{-1}.
$$
\end{prop}
\begin{proof}
Substitute $x = f^{-1}(y)$ and $x' = f^{-1}(y')$ into (\ref{eq:d2}). The remaining equalities follow from substituting $y = f(x')$ into (\ref{eq:d4}).
\end{proof}
\begin{cor}\label{cor:ghi}
For metric spaces $X$ and $Y$ of the same cardinality,
\begin{align*}
md_{GH}^i(X,Y) = d_{GH}^i(X,Y) &= \frac{1}{2} \inf\{\dis f : f\in\mathcal{M}_{\mathrm{bi}}(X,Y)\} \\
&= \frac{1}{2} \inf\{\dis g : g\in\mathcal{M}_{\mathrm{bi}}(Y,X)\}.
\end{align*}
\end{cor}
\begin{prop}\label{prop:dghbi=dghi}
For metric spaces $X$ and $Y$ of the same cardinality,
$$
d_{GH}^{\mathrm{bi}}(X,Y) = md_{GH}^{\mathrm{bi}}(X,Y) = d_{GH}^i(X,Y).
$$
\end{prop}
\begin{proof}
Since $md_{GH}^{\mathrm{bi}} \le d_{GH}^{\mathrm{bi}} \le d_{GH}^i = md_{GH}^i$, it suffices to prove $d_{GH}^i \le md_{GH}^{\mathrm{bi}}$. From Corollary~\ref{cor:ghi},
\begin{align*}
md_{GH}^{\mathrm{bi}}(X,Y) &= \frac{1}{2} \inf\{\max\{\dis f, \dis g\} : f\in\mathcal{M}_{\mathrm{bi}}(X,Y),\ g\in\mathcal{M}_{\mathrm{bi}}(Y,X)\} \\
&= \frac{1}{2} \max\left\{\inf\{\dis f : f\in\mathcal{M}_{\mathrm{bi}}(X,Y)\},\ \inf\{\dis g : g\in\mathcal{M}_{\mathrm{bi}}(Y,X)\}\right\} \\
&= d_{GH}^i(X,Y).
\end{align*}
\end{proof}

\subsection{Distance Between Cardinality Homogeneous Spaces of the Same Cardinality}
A topological space $X$ is \emph{cardinality homogeneous\/} if every nonempty open subset $U\subset X$ has cardinality $|U| = |X|$.

The following refines Theorem~\ref{thm:1-1exist} for metric spaces.
\begin{prop}\label{prop:dis}
Let $X$, $Y$, $\mathcal{L}_X$, $\mathcal{L}_Y$, $f$, $g$, and $\tilde{f}$ be as in Theorem~$\ref{thm:1-1exist}$. Additionally, let $X$ and $Y$ be metric spaces and $\e_X$, $\e_Y > 0$ finite with $\diam U \le \e_X$ for all $U\in\mathcal{L}_X$ and $\diam V \le \e_Y$ for all $V\in\mathcal{L}_Y$.

For $x,x'\in X$,
$$
|f(x)\tilde{f}(x)| \le
\begin{cases}
\e_Y & \text{if $x\in X_I$}, \\
d_{(f,g)}^+ + \e_X & \text{if $x\in X_{II}$};
\end{cases}
$$
\begin{align*}
-d_f^- - 2\e_Y &\le |\tilde{f}(x)\tilde{f}(x')| - |xx'| \le d_f^+ + 2\e_Y && \text{if $x,x'\in X_I$}, \\
-d_g^+ - 2\e_X &\le |\tilde{f}(x)\tilde{f}(x')| - |xx'| \le d_g^- + 2\e_X && \text{if $x,x'\in X_{II}$}, \\
-d_{(f,g)}^- - \e_X - \e_Y &\le |\tilde{f}(x)\tilde{f}(x')| - |xx'| \le d_{(f,g)}^+ + \e_X + \e_Y && \text{if $x\in X_I$, $x'\in X_{II}$}.
\end{align*}
Consequently, with $\e = \max\{\e_X,\e_Y\}$,
\begin{align*}
-\max\{d_f^-, d_g^+, d_{(f,g)}^-\} - 2\e &\le |\tilde{f}(x)\tilde{f}(x')| - |xx'| \le \max\{d_f^+, d_g^-, d_{(f,g)}^+\} + 2\e, \\
\frac{1}{2} \dis \tilde{f} &\le d_{GH}(f,g) + \e.
\end{align*}
\end{prop}
\begin{proof}
We provide the detailed estimates for each case.

To bound $|f(x)\tilde{f}(x)|$:

If $x\in X_I$, then $f(x)$ and $\tilde{f}(x)$ lie in the same member of $\mathcal{L}_Y$, so $|f(x)\tilde{f}(x)| \le \e_Y$.

If $x\in X_{II}$, let $\tilde{g} = \tilde{f}^{-1}$ and $y = \tilde{f}(x)$. Then
$$
|f(x)y| \le |x g(y)| + d_{(f,g)}^+ \le |x \tilde{g}(y)| + \e_X + d_{(f,g)}^+ = \e_X + d_{(f,g)}^+.
$$

For $x,x'\in X_I$,
$$
|xx'| - d_f^- - 2\e_Y \le |f(x)f(x')| - 2\e_Y \le |\tilde{f}(x)\tilde{f}(x')| \le |f(x)f(x')| + 2\e_Y \le |xx'| + d_f^+ + 2\e_Y.
$$

For $x,x'\in X_{II}$,
\begin{multline*}
|xx'| - 2\e_X - d_g^+ \le |x g\tilde{f}(x)| + |g\tilde{f}(x) g\tilde{f}(x')| + |g\tilde{f}(x') x'| - 2\e_X - d_g^+ \le \\
|g\tilde{f}(x) g\tilde{f}(x')| - d_g^+ \le |\tilde{f}(x)\tilde{f}(x')| \le |g\tilde{f}(x) g\tilde{f}(x')| + d_g^- \le \\
|g\tilde{f}(x) x| + |xx'| + |x' g\tilde{f}(x')| + d_g^- \le |xx'| + 2\e_X + d_g^-.
\end{multline*}

For $x\in X_I$ and $x'\in X_{II}$,
\begin{multline*}
|xx'| - \e_X - d_{(f,g)}^- - \e_Y \le |xx'| - |x' g\tilde{f}(x')| - d_{(f,g)}^- - \e_Y \le \\
|x g\tilde{f}(x')| - d_{(f,g)}^- - \e_Y \le |f(x)\tilde{f}(x')| - |\tilde{f}(x) f(x)| \le \\
|\tilde{f}(x)\tilde{f}(x')| \le |\tilde{f}(x) f(x)| + |f(x)\tilde{f}(x')| \le \e_Y + |x g\tilde{f}(x')| + d_{(f,g)}^+ \le \\
|xx'| + |x' g\tilde{f}(x')| + d_{(f,g)}^+ + \e_Y \le |xx'| + \e_X + d_{(f,g)}^+ + \e_Y.
\end{multline*}
\end{proof}
\begin{thm}\label{cor:dis}
For cardinality homogeneous metric spaces $X$ and $Y$ of the same cardinality,
\begin{align*}
d_{GH}(X,Y) &= d_{GH}^{\mathrm{in}}(X,Y) = d_{GH}^{\mathrm{sur}}(X,Y) = d_{GH}^{\mathrm{bi}}(X,Y) = d_{GH}^i(X,Y) \\
&= \frac{1}{2} \inf\{\dis f : f\colon X\to Y\ \text{is a bijection}\}.
\end{align*}
\end{thm}
\begin{proof}
It suffices to consider $d_{GH}(X,Y) < \infty$. From the inequalities in Section~\ref{ssec:peq}, Corollary~\ref{cor:ghi}, and Proposition~\ref{prop:dghbi=dghi}, we need only show $d_{GH}^i(X,Y) \le d_{GH}(X,Y)$.

Given $\e > 0$, choose $f\colon X\to Y$ and $g\colon Y\to X$ with $d_{GH}(f,g) < d_{GH}(X,Y) + \e$.

Cover $X$ and $Y$ by open balls of radius $\e$. By Proposition~\ref{prop:dis}, the bijection $\tilde{f}$ from Theorem~\ref{thm:1-1exist} satisfies $\frac{1}{2} \dis \tilde{f} \le d_{GH}(f,g) + \e$. By Corollary~\ref{cor:ghi}, $d_{GH}^i(X,Y) \le \frac{1}{2} \dis \tilde{f} < d_{GH}(X,Y) + \e$. Since $\e > 0$ is arbitrary, the result follows.
\end{proof}

\section{Approximation of $\dl$-Surjective Mappings}
A mapping $f\colon X\to Y$ is \emph{$\dl$-surjective\/} if $d_H(f(X),Y) < \dl$. We say $f$ maps $X$ \emph{$\dl$-surjectively\/} onto $L\subset Y$ if $d_H(f(X),L) < \dl$.

\begin{prop}\label{prop:H}
For metric spaces $X$, $Y$ and mappings $f\colon X\to Y$, $g\colon Y\to X$,
$$
d_H(f(X),Y) \le d_{(f,g)}^+.
$$
Thus, if $d_{(f,g)}^+ < \dl$, then $f$ is $\dl$-surjective.
\end{prop}
\begin{proof}
For any $y\in Y$, set $x = g(y)$. Then (\ref{eq:d4}) gives
$$
|f(x)y| \le |f(g(y)) y| = |f(g(y)) y| - |g(y)g(y)| \le d_{(f,g)}^+.
$$
Hence $d_H(f(X),Y) \le d_{(f,g)}^+$.
\end{proof}

\begin{prop}\label{prop:Hg}
For a $\dl$-surjective mapping $f\colon X\to Y$ between metric spaces, there exists $g\colon Y\to X$ such that
$$
d_g^+ \le d_f^- + 2\dl, \quad d_g^- \le d_f^+ + 2\dl, \quad d_{(f,g)}^+ \le d_f^+ + \dl, \quad d_{(f,g)}^- \le d_f^- + \dl.
$$
Consequently,
$$
d_{GH}(X,Y) \le d_{GH}(f,g) \le \frac{1}{2} \dis f + \dl.
$$
\end{prop}
\begin{proof}
Define $g$ on $f(X)$ by $g(f(x)) = x$. For $y \notin f(X)$, choose $y'\in f(X)$ with $|yy'| < \dl$ and set $g(y) = g(y')$.

Then
$$
|g(y_1)g(y_2)| = |g(y_1')g(y_2')| = |x_1' x_2'| \le |f(x_1')f(x_2')| + d_f^- = |y_1' y_2'| + d_f^- < |y_1 y_2| + 2\dl + d_f^-,
$$
and similarly
$$
|g(y_1)g(y_2)| = |x_1' x_2'| \ge |f(x_1')f(x_2')| - d_f^+ > |y_1 y_2| - 2\dl - d_f^+.
$$
Also,
$$
|f(x)y| < |f(x)y'| + \dl = |f(x)f(x')| + \dl \le |xx'| + d_f^+ + \dl = |x g(y')| + d_f^+ + \dl,
$$
and analogously
$$
|f(x)y| > |f(x)y'| - \dl = |f(x)f(x')| - \dl \ge |xx'| - d_f^- - \dl = |x g(y')| - d_f^- - \dl.
$$
\end{proof}

\begin{prop}\label{prop:Hg2}
If $d_{GH}(X,Y) < d$, then there exists a $d$-surjective $(2d)$-isometry $f\colon X\to Y$.
\end{prop}
\begin{proof}
There exist $f\colon X\to Y$ and $g\colon Y\to X$ with $d_{GH}(f,g) < d$. By Proposition~\ref{prop:H}, $f$ is $d$-surjective. Since $\frac{1}{2} \dis f \le d_{GH}(f,g) < d$, $f$ is a $(2d)$-isometry.
\end{proof}

From Propositions~\ref{prop:Hg} and~\ref{prop:Hg2} we obtain:
\begin{cor}\label{cor:H}
For metric spaces $X$ and $Y$, the following are equivalent:
\begin{enumerate}
\item $d_{GH}(X,Y) < \infty$;
\item there exist finite $d,\dl$ and a $\dl$-surjective $d$-isometry $f\colon X\to Y$;
\item there exist finite $d',\dl'$ and a $\dl'$-surjective $d'$-isometry $g\colon Y\to X$.
\end{enumerate}
\end{cor}

\begin{thm}\label{thm:sur}
Let $X$ and $Y$ be cardinality homogeneous metric spaces of the same cardinality. If there exists a $\dl$-surjective $d$-isometry $f\colon X\to Y$, then there exists a bijective $(d+2\dl)$-isometry $\tilde{f}\colon X\to Y$.
\end{thm}
\begin{proof}
By Proposition~\ref{prop:Hg}, there exists $g\colon Y\to X$ with $d_{GH}(X,Y) \le \frac{1}{2} \dis f + \dl < \frac{1}{2} d + \dl$. Theorem~\ref{cor:dis} yields a bijection $\tilde{f}\colon X\to Y$ with $d_{GH}(X,Y) \le \frac{1}{2} \dis \tilde{f} < \frac{1}{2} d + \dl$. Thus $\dis \tilde{f} < d + 2\dl$.
\end{proof}

\begin{prop}\label{prop:Hhat}
If $f\colon X\to Y$ maps $X$ $\dl$-surjectively onto a subset $L\subset Y$ with $\dl > 0$, then there exists $\hat{f}\colon X\to L$ with $|\hat{f}(x) f(x)| < \dl$ for all $x\in X$. Hence
$$
d_H(\hat{f}(X),L) \le 2\dl, \quad d_{\hat{f}}^+ \le d_f^+ + 2\dl, \quad d_{\hat{f}}^- \le d_f^- + 2\dl,
$$
so $\hat{f}$ is $(2\dl)$-surjective onto $L$.
\end{prop}
\begin{proof}
For each $x\in X$, choose $\hat{f}(x) \in L$ with $| \hat{f}(x) f(x) | < \dl$.
\end{proof}

\begin{prop}\label{prop:sur2}
If $f\colon X\to Y$ is a $d$-isometry that maps $X$ $\dl$-surjectively onto $L\subset Y$ ($\dl,d>0$), then there exists a $(2\dl)$-surjective $(d+2\dl)$-isometry $\hat{f}\colon X\to L$.
\end{prop}
\begin{proof}
By Proposition~\ref{prop:Hhat}, there exists $\hat{f}\colon X\to L$ that is $(2\dl)$-surjective onto $L$ and satisfies $|\hat{f}(x) f(x)| < \dl$ for all $x\in X$. Thus $\hat{f}$ is a $(d+2\dl)$-isometry.
\end{proof}

\begin{thm}\label{thm:sur2}
If $f\colon X\to Y$ is a $d$-isometry mapping $X$ $\dl$-surjectively onto $L\subset Y$ ($\dl,d>0$) and both $X$ and $L$ are cardinality homogeneous with $|X|=|L|$, then there exists a bijective $(d+6\dl)$-isometry $\tilde{f}\colon X\to L$.
\end{thm}
\begin{proof}
Proposition~\ref{prop:sur2} yields a $(2\dl)$-surjective $(d+2\dl)$-isometry $\hat{f}\colon X\to L$. Theorem~\ref{thm:sur} then gives a bijective $(d+6\dl)$-isometry $\tilde{f}\colon X\to L$.
\end{proof}

\section{Application to Banach Spaces}
\begin{prop}\label{prop:tvs}
Every topological vector space is cardinality homogeneous.
\end{prop}
\begin{proof}
Let $U$ be a nonempty open subset of a topological vector space $X$ and let $x\in U$. Then $U$ and $V = U - x = \{y - x : y\in U\}$ are equicardinal. Since $X = \bigcup_{n=1}^\infty nV$, $V$ and $X$ are equicardinal.
\end{proof}

\begin{prop}\label{prop:bshp}
If there exists a $\dl$-surjective $d$-isometry between Banach spaces $X$ and $Y$ \(for finite $\dl$ and $d$\), then $|X| = |Y|$.
\end{prop}
\begin{proof}
We show $d(X) = d(Y)$. Let $Z\subset X$ be dense with $|Z| = d(X)$. The restriction $\tilde{f} = f|_Z \colon Z\to Y$ is $\e$-surjective with $\e = d + \dl$, so $d_H(f(Z),Y) < \e$ and $f(Z)$ is an $\e$-net in $Y$. Then $\frac{1}{n} f(Z)$ is a $\frac{\e}{n}$-net, and $S = \bigcup_{n=1}^\infty \frac{1}{n} f(Z)$ is dense in $Y$ with $|S| \le |Z|$.

By Corollary~\ref{cor:H}, there exists a $\dl'$-surjective $d'$-isometry $g\colon Y\to X$ (for finite $d'$ and $\dl'$). Symmetrically, $d(X) \le d(Y)$. Thus $d(X) = d(Y)$. By Toru\'nczyk's theorem~\cite{Tor1981}, $X$ and $Y$ are homeomorphic, so $|X| = |Y|$.
\end{proof}

\begin{rk}\label{rk:DilwTabor}
In~\cite[Theorem 1]{T00}, any $\dl$-surjective mapping $f\colon V\to W$ between Banach spaces $V,W$ without low-cardinality dense subsets can be transformed into a surjection $F\colon V\to W$ with $\|f - F\| \le 7\dl$. The proof of Proposition 2 in~\cite{D99} similarly constructs a surjective $F$ with $\|f - F\| \le 2d + \dl$. This reduces the approximation of $\dl$-surjective $d$-isometries to that of $(2d + \dl)$- or $(d + 14\dl)$-isometric surjections.

Below, we construct, from a $d$-isometric $\dl$-surjective $f\colon V\to W$, a mapping $g\colon W\to V$ such that Theorem~\ref{thm:1-1exist} applies to the pair $(f,g)$. This yields a simpler derivation of the Dilworth--Tabor theorem.
\end{rk}

\begin{prop}\label{prop:bsh}
If there exists a $\dl$-surjective $d$-isometry $f\colon X\to Y$ between Banach spaces $X$ and $Y$, then there exists a bijective $(d+2\dl)$-isometry $\tilde{f}\colon X\to Y$.
\end{prop}
\begin{proof}
Proposition~\ref{prop:tvs} implies that $X$ and $Y$ are cardinality homogeneous, and Proposition~\ref{prop:bshp} gives $|X|=|Y|$. The result follows from Theorem~\ref{thm:sur}.
\end{proof}

Combining Proposition~\ref{prop:bsh} with the theorem of M.~Omladi\v{c} and P.~\v{S}emrl~\cite{OS95} yields:
\begin{cor}[Tabor~\cite{T00}]\label{cor:Tabor}
If there exists a $\dl$-surjective $d$-isometry $f\colon X\to Y$ between Banach spaces, then there exists a bijective affine isometry $U\colon X\to Y$ with $\|f - U\| \le 2d + 4\dl$.
\end{cor}

S.\,J.~Dilworth proved his result for mappings whose image is close to a closed subspace rather than the whole space. The same technique applies here.

\begin{prop}\label{prop:bshL}
If $f\colon X\to Y$ is a $d$-isometry between Banach spaces that maps $X$ $\dl$-surjectively onto a closed linear subspace $L\subset Y$, then there exists a bijective $(d+6\dl)$-isometry $\tilde{f}\colon X\to L$.
\end{prop}
\begin{proof}
By Proposition~\ref{prop:tvs}, $X$ and $L$ are cardinality homogeneous. Proposition~\ref{prop:sur2} gives a $(2\dl)$-surjective $(d+2\dl)$-isometry $\hat{f}\colon X\to L$. Proposition~\ref{prop:bshp} yields $|X|=|L|$. Apply Theorem~\ref{thm:sur}.
\end{proof}

Combining Proposition~\ref{prop:bshL} with~\cite{OS95} yields:
\begin{cor}[Dilworth~\cite{D99}]\label{cor:bshL}
If $f\colon X\to Y$ is a $d$-isometry between Banach spaces mapping $X$ $\dl$-surjectively onto a closed linear subspace $L\subset Y$, then there exists a bijective affine isometry $U\colon X\to L$ with $\|f - U\| \le 2d + 12\dl$.
\end{cor}

\begin{rk}\label{rk:SViBT}
The theorem of \v{S}emrl--V\"{a}is\"{a}l\"{a}~\cite[Theorem 3.2]{SV03} improves the bound in Corollary~\ref{cor:bshL} to $\|f - U\| \le 2d + 2\dl$, which is sharp. Combining Corollaries~\ref{cor:H} and~\ref{cor:Tabor} shows that two Banach spaces are isometric if and only if their Gromov--Hausdorff distance is finite~\cite[Theorem 0.3]{BT26.2}.
\end{rk}

\markright{References}

\end{document}